\pdfoutput=1
\PassOptionsToPackage{hidelinks,linktocpage,colorlinks=true,allcolors=blue,linktoc=all}{hyperref}
\PassOptionsToPackage{capitalise,nameinlink,noabbrev}{cleveref}

\newif\ifusecolttemplate
\IfFileExists{includes/colt2025.cls}{%
  \IfFileExists{algorithm2e.sty}{\usecolttemplatetrue}{\usecolttemplatefalse}%
}{\usecolttemplatefalse}
\ifusecolttemplate
  \documentclass[final,cleveref]{includes/colt2025}
  \jmlrpages{} %
  \crefname{equation}{}{}
\else
  \documentclass[11pt]{article}
\fi

\usepackage{amsmath,amssymb,mathtools,mathrsfs}
\ifusecolttemplate\else
  \usepackage{amsthm}
\fi
\usepackage{xparse}
\usepackage{xargs}
\usepackage{xcolor}
\usepackage{tikz}
\usepackage{enumitem}
\usepackage{booktabs}
\usepackage{aliascnt}
\usepackage{url}
\ifusecolttemplate\else
  \usepackage[margin=1.05in]{geometry}
  \usepackage[colorlinks=true,linkcolor=blue,citecolor=blue,urlcolor=blue]{hyperref}
  \usepackage[capitalise,nameinlink,noabbrev]{cleveref}
\fi

\usepackage{algorithm}
\usepackage{algcompatible}
\algnewcommand{\lst}{\texttt{lst}}
\algnewcommand{\slst}{\texttt{slst}}
\algnewcommand{\SEND}{\textbf{send}}

\newsavebox{\algleft}
\newsavebox{\algright}

\makeatletter
\newcounter{algorithmicH}
\let\oldalgorithmic\algorithmic
\renewcommand{\algorithmic}{%
  \stepcounter{algorithmicH}
  \oldalgorithmic}
\renewcommand{\theHALG@line}{ALG@line.\thealgorithmicH.\arabic{ALG@line}}
\makeatother

\makeatletter
\newenvironment{breakablealgorithm}
  {
   \begin{center}
     \refstepcounter{algorithm}
     \hrule height.8pt depth0pt \kern2pt
     \renewcommand{\caption}[2][\relax]{
       {\raggedright\textbf{\ALG@name~\thealgorithm} ##2\par}%
       \ifx\relax##1\relax 
         \addcontentsline{loa}{algorithm}{\protect\numberline{\thealgorithm}##2}%
       \else 
         \addcontentsline{loa}{algorithm}{\protect\numberline{\thealgorithm}##1}%
       \fi
       \kern2pt\hrule\kern2pt
     }
  }{
     \kern2pt\hrule\relax
   \end{center}
  }

\makeatother

\usepackage[natbib, backend=biber, maxcitenames=3, minalphanames=3, maxbibnames=99, style=alphabetic, hyperref, backref, useprefix=true, uniquename=false, doi=false,url=false,eprint=false]{biblatex} %

\usepackage{csquotes}               %
\bibliography{refs}        %

\DeclareCiteCommand{\cite}
  {\usebibmacro{prenote}}
  {\usebibmacro{citeindex}%
   \printtext[bibhyperref]{\usebibmacro{cite}}}
  {\multicitedelim}
  {\usebibmacro{postnote}}

\DeclareCiteCommand*{\cite}
  {\usebibmacro{prenote}}
  {\usebibmacro{citeindex}%
   \printtext[bibhyperref]{\usebibmacro{citeyear}}}
  {\multicitedelim}
  {\usebibmacro{postnote}}

\DeclareCiteCommand{\parencite}[\mkbibparens]
  {\usebibmacro{prenote}}
  {\usebibmacro{citeindex}%
    \printtext[bibhyperref]{\usebibmacro{cite}}}
  {\multicitedelim}
  {\usebibmacro{postnote}}

\DeclareCiteCommand*{\parencite}[\mkbibparens]
  {\usebibmacro{prenote}}
  {\usebibmacro{citeindex}%
    \printtext[bibhyperref]{\usebibmacro{citeyear}}}
  {\multicitedelim}
  {\usebibmacro{postnote}}

\DeclareCiteCommand{\citeauthor}
  {\usebibmacro{prenote}}
  {\ifciteindex
     {\indexnames{labelname}}
     {}%
   \printtext[bibhyperref]{\printnames{labelname}}}
  {\multicitedelim}
  {\usebibmacro{postnote}}

\DeclareCiteCommand{\footcite}[\mkbibfootnote]
  {\usebibmacro{prenote}}
  {\usebibmacro{citeindex}%
  \printtext[bibhyperref]{ \usebibmacro{cite}}}
  {\multicitedelim}
  {\usebibmacro{postnote}}

\DeclareCiteCommand{\footcitetext}[\mkbibfootnotetext]
  {\usebibmacro{prenote}}
  {\usebibmacro{citeindex}%
   \printtext[bibhyperref]{\usebibmacro{cite}}}
  {\multicitedelim}
  {\usebibmacro{postnote}}

\DeclareCiteCommand{\textcite}
  {\boolfalse{cbx:parens}}
  {\usebibmacro{citeindex}%
   \printtext[bibhyperref]{\usebibmacro{textcite}}}
  {\ifbool{cbx:parens}
     {\bibcloseparen\global\boolfalse{cbx:parens}}
     {}%
   \multicitedelim}
  {\usebibmacro{textcite:postnote}}

\newbibmacro{string+doiurlisbn}[1]{%
  \iffieldundef{doi}{%
    \iffieldundef{url}{%
      \iffieldundef{isbn}{%
        \iffieldundef{issn}{%
          #1%
        }{%
          \href{http://books.google.com/books?vid=ISSN\thefield{issn}}{#1}%
        }%
      }{%
        \href{http://books.google.com/books?vid=ISBN\thefield{isbn}}{#1}%
      }%
    }{%
      \href{\thefield{url}}{#1}%
    }%
  }{%
    \href{https://doi.org/\thefield{doi}}{#1}%
  }%
}

\DeclareFieldFormat{title}{\usebibmacro{string+doiurlisbn}{\mkbibemph{#1}}}
\DeclareFieldFormat[article,incollection,inproceedings]{title}%
    {\usebibmacro{string+doiurlisbn}{#1}}
\input{includes/definitions.tex}%
\let\paperLambda\Lambda
\newcommand{\notationcommand}[3]{%
  \providecommand{#1}{}%
  \renewcommand{#1}{\newlink{#2}{#3}}%
}

\let\oldell\ell
\let\oldmu\mu
\let\oldeta\eta
\let\oldtau\tau
\let\oldGamma\Gamma
\let\oldDelta\Delta
\let\oldomega\omega
\let\oldrho\rho
\let\oldzeta\zeta
\let\oldpartial\partial
\let\olddelta\delta
\let\oldTheta\Theta
\let\oldOmega\Omega

\notationcommand{\p}{def:p}{p}
\notationcommand{\q}{def:q}{q}
\notationcommand{\d}{def:dimension}{d}
\notationcommand{\RR}{def:radius}{R}
\notationcommand{\G}{def:lipschitz}{G}
\notationcommand{\T}{def:horizon}{T}
\notationcommand{\eps}{def:accuracy}{\varepsilon}
\notationcommand{\delta}{def:confidence}{\olddelta}
\notationcommand{\f}{def:objective}{f}

\notationcommand{\g}{def:subgradient}{g}
\notationcommand{\Q}{def:domain}{\mathcal Q}
\notationcommand{\B}{def:ball}{B}
\notationcommand{\Baff}{def:affine-ball}{B}
\notationcommand{\ell}{def:norm}{\oldell}
\newcommand{\conj}{\newlink{def:conjugate}{\ast}}
\notationcommand{\partial}{def:subgradient}{\oldpartial}
\notationcommand{\R}{def:basic-notation}{\mathbb R}
\notationcommand{\E}{def:probability}{\mathbb E}
\notationcommand{\Prb}{def:probability}{\mathbb P}
\notationcommand{\cN}{def:gaussian}{\mathcal N}
\notationcommand{\Z}{def:gaussian}{Z}
\notationcommand{\Id}{def:gaussian}{I}
\notationcommand{\Sphere}{def:sphere}{\mathbb S}
\notationcommand{\sphsigma}{def:sphere}{\sigma}
\notationcommand{\pospart}{def:positive-part}{+}
\notationcommand{\1}{def:indicator}{\mathbf 1}
\notationcommand{\dist}{def:distance}{\operatorname{dist}}
\renewcommand{\norm}[1]{\newlink{def:norm}{\left\lVert #1\right\rVert}}
\renewcommand{\ip}[2]{\newlink{def:inner-product}{\left\langle #1,#2\right\rangle}}

\notationcommand{\loss}{def:online-loss}{\oldell}
\notationcommand{\epssign}{def:online-signs}{\varepsilon}
\notationcommand{\Reg}{def:regret}{\operatorname{Reg}}
\notationcommand{\xonline}{def:online-iterates}{x}
\notationcommand{\xchase}{def:chasing-iterates}{x}
\notationcommand{\xquery}{def:query-points}{x}
\notationcommand{\ebasis}{def:standard-basis}{e}
\notationcommand{\Nphase}{def:phase-horizon}{N}
\notationcommand{\stail}{def:head-tail-exponent}{s}
\notationcommand{\K}{def:requests}{K}
\notationcommand{\z}{def:selector-points}{z}
\notationcommand{\zhat}{def:approximate-center}{\widehat z}
\notationcommand{\model}{def:bundle-model}{m}
\notationcommand{\lb}{def:lower-bound}{\oldell}
\notationcommand{\U}{def:upper-bound}{U}
\notationcommand{\Gamma}{def:phase-scale}{\oldGamma}
\newcommand{\hminus}{\newlink{def:lower-level}{h}_{-}}
\newcommand{\hplus}{\newlink{def:upper-level}{h}_{+}}

\notationcommand{\oracle}{def:oracle}{\mathfrak O}
\notationcommand{\selector}{def:selector-rule}{\mathsf C}

\notationcommand{\xhat}{def:output}{\widehat x}
\notationcommand{\rho}{def:reduction-exponent}{\oldrho}

\notationcommand{\supporth}{def:support-function}{h}
\notationcommand{\st}{def:steiner}{\operatorname{st}}
\notationcommand{\omega}{def:mean-width}{\oldomega}
\notationcommand{\shat}{def:empirical-steiner}{\widehat s}
\notationcommand{\Nsample}{def:sample-count}{N}
\notationcommand{\errE}{def:sampling-error}{E}
\notationcommand{\tau}{def:threshold}{\oldtau}

\notationcommand{\Rtail}{def:tail-radius}{R}
\notationcommand{\rmin}{def:tail-exponent}{r}
\notationcommand{\snear}{def:head-exponent}{s}

\notationcommand{\cL}{def:lift}{\mathcal L}
\notationcommand{\mu}{def:mu}{\oldmu}
\notationcommand{\eta}{def:eta}{\oldeta}
\notationcommand{\gdir}{def:perturbation}{g}
\notationcommand{\usample}{def:sample-optimizer}{u}
\notationcommand{\vsample}{def:sample-optimizer}{v}
\notationcommand{\F}{def:sample-objective}{F}

\notationcommand{\D}{def:bregman}{D}

\notationcommand{\Chat}{def:empirical-center}{\widehat C}
\notationcommand{\Y}{def:center-sample}{Y}
\notationcommand{\Yhat}{def:numerical-sample}{\widehat Y}

\notationcommand{\mval}{def:sample-value}{m}
\notationcommand{\Delta}{def:value-increment}{\oldDelta}
\notationcommand{\ubar}{def:mean-components}{\bar u}
\notationcommand{\vbar}{def:mean-components}{\bar v}
\notationcommand{\zeta}{def:objective-tolerance}{\oldzeta}
\notationcommand{\Lambda}{def:log-factor}{\oldGamma}

\notationcommand{\bigO}{def:asymptotics}{O}
\notationcommand{\bigOtilde}{def:soft-asymptotics}{\widetilde O}
\notationcommand{\Theta}{def:asymptotics}{\oldTheta}
\notationcommand{\Omega}{def:asymptotics}{\oldOmega}
\notationcommand{\Thetatilde}{def:soft-asymptotics}{\widetilde\oldTheta}

\renewcommand{\bigopl}[2]{\bigO_{#1}\left(#2\right)}

\renewcommand{\bigotildepl}[2]{\bigOtilde_{#1}\left(#2\right)}

\hypersetup{colorlinks=true,allcolors=blue,linktoc=all,hypertexnames=false}
\renewcommand{\Lambda}{\paperLambda}
\renewcommand{\nu}{\newlink{def:numerical-tolerance}{\oldnu}}

\ifusecolttemplate
  \newtheorem{assumption}[theorem]{Assumption}
\else
  \theoremstyle{plain}
  \newtheorem{theorem}{Theorem}[section]
  \newtheorem{lemma}[theorem]{Lemma}
  \newtheorem{proposition}[theorem]{Proposition}
  
  \theoremstyle{definition}
  \newtheorem{definition}[theorem]{Definition}

  \newcommand{\coltauthor}[1]{\author{#1}}
\fi

\crefname{algorithm}{Algorithm}{Algorithms}
\crefname{assumption}{Assumption}{Assumptions}
\crefname{proposition}{Proposition}{Propositions}
\crefname{corollary}{Corollary}{Corollaries}
\crefname{lemma}{Lemma}{Lemmas}
\crefname{theorem}{Theorem}{Theorems}
\crefname{definition}{Definition}{Definitions}
\crefname{remark}{Remark}{Remarks}

\newcommand{\Center}{\mathsf C}
\newcommand{\Step}{\mathsf{Step}}
\newcommand{\Dholder}{\mathcal D_\kappa}
\newcommand{\Ncone}{N}
\newcommand{\Menv}{M_A}

\newcommand{\authorasterisk}{\textsuperscript{\normalfont *}}

\ifusecolttemplate
  \title[Near-Optimal Smooth Nondual Convex Optimization]{Near-Optimal Acceleration for Smooth \texorpdfstring{$\ell_p/\ell_q$}{lp/lq} Nondual Convex First-Order Oracle Optimization}
\else
  \title{Near-Optimal Acceleration for Smooth \texorpdfstring{$\ell_p/\ell_q$}{lp/lq} Nondual Convex First-Order Oracle Optimization}
\fi
\coltauthor{
\Name{David Martínez-Rubio}\Email{\href{mailto:david.martinezrubio@imdea.org}{david.martinezrubio@imdea.org}}\\
\addr IMDEA Software Institute, Madrid, Spain
\AND
\Name{Brian Bullins\nametag{\authorasterisk}}\Email{\href{bbullins@purdue.edu}{bbullins@purdue.edu}}\\
\addr Purdue University, West Lafayette, IN, USA
\AND
\Name{Cristóbal Guzmán\nametag{\authorasterisk}}\Email{\href{mailto:crguzmanp@uc.cl}{crguzmanp@uc.cl}}\\
\addr Institute for Mathematical and Computational Engineering, Faculty of Mathematics and
School of Engineering, Pontificia Universidad Católica de Chile, Santiago, Chile
\AND
\Name{Mathieu Molina\nametag{\authorasterisk}} \Email{\href{mailto:mathieu.molina.research@gmail.com}{mathieu.molina.research@gmail.com}}\\
\addr Tel Aviv University, Israel
}
\date{}

\newcommand\blfootnote[1]{%
  \begingroup
  \renewcommand\thefootnote{}\footnote{#1}%
  \addtocounter{footnote}{-1}%
  \endgroup
}

\begin{document}
\maketitle

\begin{abstract}
    We study the optimization of convex objectives with $(L,\kappa-1)$-H\"older-continuous gradients in $\ell_q$ over $R B_p^d$, $1<\kappa\le2$. \citet{martinezRubioGuzman2026stable} provides selectors with a movement bound for the problem of chasing high-dimensional convex nested sets for every $p<q$ and generally reduces Lipschitz convex optimization to bounds on the movement of selectors. We couple that movement with H\"older descent yielding a polynomial-runtime first-order method whose feasible output, in the high-dimensional regime $T\le d$ and for
$p<\min\{q,2\}$, has error
\begin{equation*}
  \widetilde O_{\kappa,p,q}\!\left(
    \frac{LR^\kappa}{T^{\kappa(1+1/p-(1/q-1/2)_+)-1}}
  \right),
\end{equation*}
after $T$ queries to a first-order oracle, solving the COLT 2015 open problem of \citet{guzman2015open}, up to logarithmic factors. At $(p,q)=(1,2)$, the rate is $\widetilde O(LR^\kappa/T^{2\kappa-1})$, including $\widetilde O(LR^2/T^{3})$ cubic decay in the smooth case.

\end{abstract}

\blfootnote{This draft is not yet in the form in which we would have liked to share it. In light of recent developments, we have nevertheless decided to make it available now. We plan to polish and revise this work shortly.}
\begingroup
\renewcommand{\thefootnote}{*}
\footnotetext{Authors marked with an asterisk are listed alphabetically.}
\endgroup

\section{Introduction and result}

The open problem of \citet{guzman2015open} asks for the minimax risk of
black-box convex optimization when the domain is an $\ell_p$ ball but
smoothness is measured in a different $\ell_q$ norm.  The mismatch matters
most for $p<\min\{q,2\}$: the feasible set is smaller than the natural $\ell_q$ ball,
yet classical accelerated methods use only $\ell_p$ or $\ell_q$ geometry and do not match known lower bounds.  For instance at $(p,q)=(1,2)$ and Lipschitz gradient smoothness, the classical rate is
$O(T^{-2})$ while the lower bound predicts $O(T^{-3})$.

The nonsmooth endpoint of this question was solved in \citep{martinezRubioGuzman2026stable} by reducing a level bundle method to movement of nested convex bodies in high dimensions. This paper shows that the same movement primitive
also supplies the missing acceleration for every $1<\kappa\le2$.  The
mechanism is not an ordinary acceleration wrapper.  A descent step contributes to a reduction in function value in the form of a positive power of the gradient norm, while a not-small-enough evaluation produces a deep cut in the bundle sublevel set and forces movement inversely proportional to that norm. The means inequalities balances the two effects, yielding acceleration, after a coupling of the movement selector and the descent point.

More generally, our reduction takes as input a feasible selector with
the required movement bound and inherits polynomial running time when that
selector is computable in polynomial time on the generated bodies.
\citet{martinezRubioGuzman2026stable} supplies the selectors and their implementation. The contribution here is the conversion of movement into H\"older acceleration, including the coupling that preserves
feasibility for constrained optimization.

If the problem is unconstrained and a global minimizer $x^\ast$ is known to satisfy $\norm{x^\ast}_p\le R$, we can provide a simple analysis of an algorithm that runs the movement selector on the auxiliary body $R B_p^d$, while descent steps and the final output are allowed to fall outside and returns a point outside it.  For the full solution, we want to return a feasible point, and we would not assume there is a zero
gradient in the set. After presenting our solution in the simplified setting, the second half of the paper constructs a solution based on approximating a Moreau envelope which, up to handling errors, reduces the problem to the previous case.

\subsection{Setting}

\newtarget{def:norm}{}\newtarget{def:conjugate}{}\newtarget{def:basic-notation}{}\newtarget{def:inner-product}{}
We use $1/\infty=0$. For $1\le r\le\infty$, the conjugate exponent
$r^\ast$ satisfies $1/r+1/r^\ast=1$, and
$B_r^d=\{x\in\mathbb R^d:\|x\|_r\le1\}$. We write $\langle u,v\rangle=\sum_{i=1}^d u_iv_i$ for the standard inner product.

\newtarget{def:q}{}
Assume
\begin{equation}\label{eq:holder-smoothness}
  \norm{\nabla f(v)-\nabla f(u)}_{q^\ast}
  \le L\norm{v-u}_q^{\kappa-1},
  \qquad 1<\kappa\le2.
\end{equation}
Integrating the gradient on a segment gives
\begin{equation}\label{eq:holder-descent-lemma}
  f(v)
  \le f(u)+\ip{\nabla f(u)}{v-u}
  +\frac L\kappa\norm{v-u}_q^\kappa.
\end{equation}
\newtarget{def:reduction-exponent}{}
For $p<q$, define the movement exponent
\begin{equation}\label{eq:rho-holder}
  \rho_{p,q}
  \defi
  \frac1p-\left(\frac1q-\frac12\right)_+.
\end{equation}
Throughout, $\widetilde O_{\kappa,p,q}$ omits constants depending on fixed
regularity and norm parameters and factors polynomial in logarithms of
$d,T$, the confidence level, and the requested numerical accuracy; it never
omits a polynomial in $d$.

\newtarget{def:domain}{}\newtarget{def:objective}{}\newtarget{def:confidence}{}
\begin{theorem}[Nondual H\"older accelerated optimization]
\label{thm:holder-main}\linktoproof{thm:holder-main}
Let $R,L>0$, let $\Q=R B_p^d$, let $1\le p,q\le\infty$, and let
$f$ be convex and differentiable on a neighborhood of $\Q$, satisfying
\cref{eq:holder-smoothness} for some $1<\kappa\le2$.  Fix
    $\delta\in(0,1)$.  For $T\le d$, and either $p<\min\{q,2\}$ or $p=q\le2$, there is a polynomial-time algorithm in the real-arithmetic model, making at most $T$ first-order oracle queries
and returning $\widehat x_T\in\Q$ with probability at least $1-\delta$ such that 
\begin{equation}\label{eq:holder-main-rate}
  f(\widehat x_T)-\min_{u\in\Q}f(u)
  \le
    \widetilde O_{\kappa,p,q}\!\left(
      \frac{LR^\kappa}
      {T^{\kappa(1+1/p-(1/q-1/2)_+)-1}}
    \right),
\end{equation}
\end{theorem}
Polynomial running time is understood in the real-arithmetic model of \cref{sec:implementation}, with the selector implementation verified in \cref{prop:selector-implementation}, with fixed regularity and norm parameters. We note that there is also a deterministic algorithm if we use the deterministic stable selectors in \citet{martinezRubioGuzman2026stable}. We do not claim efficiency for that deterministic algorithm but we note that \citet{martinezRubioGuzman2026stable} showed lower bounds of essentially the same order against randomized algorithms in comparison to the ones in the open question \citep{guzman2015open} for the Lipschitz problem.

The rate above, along with the observation that for $2 \leq p < q$ one can run the classical modified Nesterov method in the $p$-geometry \citep[Section~4.1.B]{guzmanNemirovski2015lower} to match the lower bounds in \citet{guzmanNemirovski2015lower} up to log factors, solves the COLT 2015 open problem \citep{guzman2015open}. The randomized implementation attains the same upper bound; we do not infer a randomized lower bound from the deterministic one.

\paragraph{Related optimization methods.}
The diagonal $p=q$ smooth rates go back to classical work
\citep{nemirovskiiNesterov1985}, while universal methods cover unknown
H\"older regularity \citep{nesterov2015universal}.  Our algorithm draws on level and bundle methods
\citep{lemarechal1995new} as realized by \citep{martinezRubioGuzman2026stable}. The new point is the coupling to the
nonstandard movement geometry and the feasible envelope construction.

An independent work \citep{ouyang2026algorithm} appeared on arXiv as we prepared the writing of our results. The work studies the specific case of quadratics when $(\p, \q)=(1,2)$ for the smooth case ($\kappa=2$).

\section{Warm-up: power descent versus harmonic movement}\label{sec:warmup}

In order to understand the solution, it is convenient to understand how the non-smooth rate of \citep{martinezRubioGuzman2026stable} is achieved. For intuition, we include \cref{fig:bundle-intuition}, borrowed from the paper.

\begin{figure}[h!]
  \centering
  \begin{tikzpicture}[
      x=1cm,
      y=.85cm,
      >=stealth,
      every node/.style={font=\small}
    ]
    \draw[->, black!65] (0.65,0.35) -- (11.45,0.35);
    \draw[->, black!65] (0.8,0.2) -- (0.8,5.65);
    \draw[densely dashed, black!45] (0.8,2.25) -- (11.3,2.25);
    \draw[densely dashed, black!45] (0.8,4.0) -- (11.3,4.0);
    \node[anchor=east] at (0.7,2.25) {$\hminus$};
    \node[anchor=east] at (0.7,4.0) {$\hplus$};

    \draw[densely dotted, black!50] (0.9,5.36) -- (5.05,0.42);
    \draw[densely dotted, black!50] (1.0,1.24) -- (11.0,2.14);
    \draw[densely dotted, blue!55!black] (5.1,0.42) -- (10.6,5.48);
    \draw[very thick, black!75] (0.9,5.36) -- (4.1,1.55);
    \draw[very thick, black!75] (4.1,1.55) -- (6.6,1.80);
    \draw[very thick, blue!65!black] (6.6,1.80) -- (10.6,5.48);
    \node[black!65, anchor=east] at (1.35,4.55) {$a_0$};
    \node[black!65, anchor=south] at (5.15,1.72) {$a_1$};
    \node[black, anchor=west] at (10.25,4.85) {$a_2$};

    \draw[gray!65, line width=2.6pt] (3.51,2.25) -- (11.05,2.25);
    \draw[black, line width=3.2pt] (3.51,2.25) -- (7.09,2.25);
    \node[anchor=south] at (5.30,2.30) {$\K_{3}$};
    \node[text=gray!70, anchor=north] at (8.75,1.70) {old $K_2$};

    \coordinate (querypoint) at (9.70,4.65);
    \fill[blue!65!black] (querypoint) circle (2pt);
    \node[black, anchor=east] at (9.48,4.72)
      {$\f(\z_2)$};
    \draw[densely dotted, black!55] (9.70,0.35) -- (querypoint);
    \node[anchor=north] at (9.70,0.31) {$\z_2$};

    \draw[densely dotted, black!55] (7.09,2.25) -- (7.09,2.82);
    \draw[<->, semithick] (7.12,2.72) -- (9.67,2.72)
      node[midway, fill=white, inner sep=1.5pt] {$\ge\Gamma/(2\G)$};
  \end{tikzpicture}
    \caption{A one-dimensional section of a step where the function value is no less than $h_+$ (null step).  The old model
    $\max\{a_0,a_1\}$ lies below $\hminus$ at $\z_2$ (that is, $z_2 \in K_2$), whereas the new tangent
  satisfies $a_2(\z_2)=\f(\z_2)>\hplus$.  Adding $a_2$ reduces the
    $(\hminus)$-sublevel to the thick black segment ($K_{3} \subset K_{2}$), separated from
  $\z_2$ by at least $\Gamma/(2\G)$ as the subgradient norm is no larger than $G$. Recall that $\Gamma$ is the duality gap at the beginning of a phase.}
  \label{fig:bundle-intuition}
\end{figure}
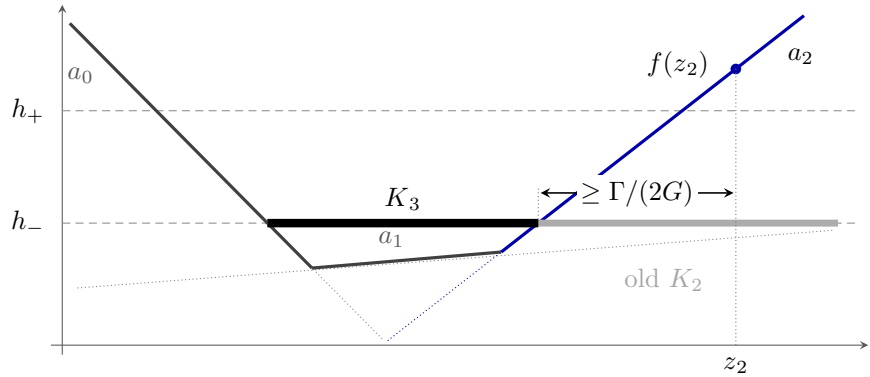

\newtarget{def:bundle-model}{}\newtarget{def:query-points}{}\newtarget{def:subgradient}{}\newtarget{def:lower-bound}{}\newtarget{def:upper-bound}{}
\paragraph{Non-smooth solution from \citep{martinezRubioGuzman2026stable}} Before query $t\ge1$, let
\[
  \model_t(x)=\max_{0\le i<t}\{\f(\xquery_i)+\ip{\g_i}{x-\xquery_i}\},
  \qquad
  \lb_t=\min_{x\in\Q}\model_t(x),
  \qquad
  \U_t=\min_{0\le i<t}\f(\xquery_i).
\]
Convexity gives $\lb_t\le \min_{x\in\Q} f(x)\le\U_t$.  We run the algorithm in phases. Freeze
\newtarget{def:phase-scale}{}\newtarget{def:lower-level}{}\newtarget{def:upper-level}{}
the phase scale $\Gamma=(\U-\lb)/2$ and two levels
$\hminus=\lb+\Gamma/2<\hplus=\lb+3\Gamma/2$, and query a stable selector
\newtarget{def:requests}{}\newtarget{def:selector-points}{}
$\xquery_t=\z_t\in\K_t:=\{x\in\Q:\model_t(x)\le\hminus\}$.  If
$\f(\z_t)\le\hplus$, the upper bound improves by a constant fraction.  Otherwise
the new tangent $a_t(x)=\f(\z_t)+\ip{\g_t}{x-\z_t}$ gives
$\K_{t+1}=\K_t\cap\{a_t\le\hminus\}$.  This retained halfspace is at $\ell_{\q}$-distance \newtarget{def:lipschitz}{} at least $(\hplus-\hminus)/\norm{\g_t}_{\q^{\conj}}=\Gamma/\norm{\g_t}_{\q^{\conj}}$ from $\z_t$.  Thus every such \emph{deep cut} forces movement of the selector, unless the new body is empty;
see \cref{fig:bundle-intuition}.
An upper bound on the total movement $\sum_{t=1}^T\norm{z_t-z_{t+1}}_q$ after $T$ steps coming from the chasing convex bodies problem along with the fact that the movement is lower bounded by $T \cdot \Gamma/\G$ gives a maximum number of steps until we see an empty sublevel set (so the lower bound increases to at least $h_{-}$) or we have $f(z_{t}) \leq h_{+}$. Either way, we decrease the gap by a constant factor and continue to the next phase.

\paragraph{Hölder-smooth warm-up}

In this warm-up, suppose the function is convex and Hölder smooth globally but its global minimum is in $R B_p^d$ while we initialized at $0$, and we allow descent points $y_t$ and the points $x_t$ that we use to query the first-order oracle to be outside of $RB_p^d$ while we have a stable selector $z_t \in RB_p^d$. We assume being able to perform exact segment function value minimization in this warm-up, for simplicity. The algorithm has some similarities to the nonsmooth solution: we run the bundle method, except that within a phase, we query at $x_t\in\argmin_{x\in[y_t,z_t]}f(x)$ receiving $f(x_t)$ and $g_t=\nabla f(x_t)$, then descend to $y_{t+1}$ from $x_t$ optimizing the unconstrained upper bound \cref{eq:holder-descent-lemma} on the function that Hölder smoothness gives while $z_t$ is the proposed stable selector. By optimality, the first-order oracle query at $x_t$ also either has function value at most $h_{+}$ or it produces a deep cut separating $z_t$ from the next sublevel set.

Line search gives $f(x_t)\le f(y_t)$ and $\ip{g_t}{z_t-x_t}\ge0$: descent improves the old value,
and the tangent $a_t(u)=f(x_t)+\ip{g_t}{u-x_t}$ satisfies $a_t(z_t)\ge f(x_t)$.
For the smooth method, replace the fixed levels by the moving target below. If descent guarantees decrease $D_t$, set $U_{t+1}=f(x_t)-D_t$, add $a_t$ to the bundle,
and lower all cuts' target to $U_{t+1}-\Gamma$, with phase scale $\Gamma$.
Convexity preserves all points of $RB_p^d$ below this target level; the bodies stay nested.
Thus the same gradient gives decrease $D_t$ and forces selector movement at least
$(\Gamma+D_t)/\norm{g_t}_{q^\ast}$, unless the next body is empty. In particular, the deep cut at $x_t$ is also a deep cut for $z_t$ and we force movement similarly as if we had queried at $z_t$ instead.

For $p<q$, \citet{martinezRubioGuzman2026stable} provides, for every nonempty
nested sequence $R B_p^d\supseteq K_0\supseteq\cdots\supseteq K_N$, a
selector $z_t\in K_t$ satisfying
\begin{equation*}
  \sum_{t<N}\norm{z_{t+1}-z_t}_q
  =\widetilde O_{p,q}\!\left(RN^{1-\rho_{p,q}}\right),
\end{equation*}
where $\rho_{p,q}$ is defined in \cref{eq:rho-holder}.

Let $g$ be a gradient and $G=\norm{g}_{q^\ast}$.  Choose a support vector
$v_q(g)$, i.e.~such that  $\norm{v_q(g)}_q=1$ and
$\ip{g}{v_q(g)}=G$.  Minimizing the H\"older upper model along
$-v_q(g)$ gives the step length and guaranteed decrease in function value $\Dholder(G)$, where
\begin{equation*}
  a=(G/L)^{1/(\kappa-1)}, \qquad
  \Dholder(G) \defi \frac{\kappa-1}\kappa L^{-1/(\kappa-1)}G^{m_\kappa}, \text{ for } m_\kappa=\frac\kappa{\kappa-1}.
\end{equation*}
Thus gradient descent controls the $m_\kappa$-power mean of the gradient norms.

If $G_t=0$, the queried point is a global minimizer and we stop. Otherwise, a tangent cut with level violation $\Gamma$ and subgradient norm $G_t$ has depth
$\Gamma/G_t$.  Hence a purely movement-based estimate controls
\begin{equation*}
  \Gamma\sum_{t<N}\frac1{G_t},
\end{equation*}
or equivalently the harmonic mean of the subgradient norms.  The inequality between
the harmonic mean and the $m_\kappa$-power mean is
\begin{equation}\label{eq:harmonic-power-mean}
  \frac{N}{\sum_{t<N}1/G_t}
  \le
  \left(\frac1N\sum_{t<N}G_t^{m_\kappa}\right)^{1/m_\kappa}.
\end{equation}

If one phase has scale $\Gamma$, the concatenated guaranteed descent is at most $U_0-f^\star\le U_0-\ell_0=2\Gamma$. For $N$ nonterminal transitions in this phase, descent gives
\begin{equation*}
  \sum_{t=0}^{N-1} G_t^{m_\kappa}
  \lesssim_\kappa L^{1/(\kappa-1)}\Gamma,
\end{equation*}
whereas movement gives
\begin{equation*}
  \Gamma\sum_{t=0}^{N-1}\frac1{G_t}
  =\widetilde O\!\left(RN^{1-\rho}\right).
\end{equation*}
Substituting these estimates into \cref{eq:harmonic-power-mean} and using
$m_\kappa=\kappa/(\kappa-1)$ yields
\begin{equation*}
  N^{\kappa-1+\kappa\rho}
  =\widetilde O_\kappa\!\left(\frac{LR^\kappa}{\Gamma}\right).
\end{equation*}
This is the desired exponent that yields optimality up to log factors. The line search makes both estimates
simultaneously valid; the later envelope construction preserves this tradeoff
using quadratic descent and phase tuning, with controlled numerical errors
and feasible queries.

To sum up, we know from previous work that we have certain movement bound $\widetilde{O}(RN^{1-\rho})$ with using some stable selectors $z_t$. We keep descent points $y_t$ from doing classical in the unconstrained problem, and we select the point $x_t$ where we will do the first-order oracle query as the minimizer of $f$ on $[y_t,z_t]$ (or an approximation with controlled errors). Optimality of the line search creates a deep cut over $z_t$ in the nested convex bodies subproblem, while ensuring $f(x_t)\le f(y_t)$. This is reminiscent of linear coupling and other accelerated regimes combining mirror descent and descent steps. Smaller gradient norms $G_t$ strengthen the guaranteed cut depth $\Gamma/G_t$, while larger norms strengthen the guaranteed descent. The means inequality balances those two phenomena. The use of this inequality is reminiscent of the intuitive explanation in \citet[Thought Experiment, p.~4]{allenZhuOrecchia2014linear} (that inspired the idea of this paper) and, in hindsight, of the use of the reverse Hölder inequality in higher-order acceleration \citep{carmonEtAl2022optimal,adilEtAl2024oracles,contrerasEtAl2024nonEuclidean}.

We note that convergence rates depending on the harmonic mean of the gradient norms is an identified phenomenon, see \citep{levy2017online,orabona2023normalized}. Our work provides an intuitive geometric explanation to this phenomenon in the form of our realized moment bound.

\section{Algorithm at a glance}

The method runs in phases. Each inner bundle call builds a tangent
model, and the outer level method retains its certified aggregate cut.
The maximum of the retained aggregate cuts forms a model $m\le f$ on
$\Q$. Minimizing $m$
gives a certified lower bound, while a retained feasible point
gives an upper bound.  Each phase fixes a fraction of this certified gap,
forms a corresponding model sublevel set, and queries the movement selector
inside that set. The coupled-step module then returns a feasible
upper witness and a new valid affine cut.

\newtarget{def:accuracy}{}
\begin{breakablealgorithm}
\caption{Feasible H\"older movement-coupled level method}
\label{alg:holder-main}
\begin{algorithmic}[1]
  \REQUIRE Feasible set $\Q=R B_p^d$; first-order oracle; H\"older
    parameters $(\kappa,L)$; accuracy $\eps$; movement selector
    $\Center$; query budget $J$ from \cref{eq:oracle-budget}.
  \REQUIRE Feasible envelope step $\Step$; initial feasible upper
    witness $y$, upper certificate $U\ge f(y)$, and max-affine model
    $m\le f$ on $\Q$.
  \ENSURE A feasible witness; its certified gap is at most $\eps$ on
    simultaneous selector success.
  \STATE Count every first-order call, including initialization and
    inner line-search calls. Before call $J+1$, return the retained $y$.
  \STATE $\ell\gets\min_{u\in\Q}m(u)$.
  \WHILE{$U-\ell>\eps$}
    \STATE $\Gamma\gets(U-\ell)/2$ and
      $U^{\mathrm{start}}\gets U$.
    \STATE Choose the planned length $N_\Gamma$ by
      \cref{eq:planned-phase-horizon}; set the envelope
      coefficient $A_\Gamma$, bundle tolerance $\tau_\Gamma$, and line-search
      tolerance $\eta_\Gamma$ by \cref{eq:phase-parameters}.
    \STATE Restart the horizon-dependent selector $\Center$ for
      $N_\Gamma$ transitions ($N_\Gamma+1$ selector points); set $n\gets0$.
    \WHILE{$U>U^{\mathrm{start}}-\Gamma/4$}
      \STATE $K\gets\set{u\in\Q:m(u)\le U-\Gamma}$.
      \IF{$K=\varnothing$}
        \STATE \textbf{break}.
      \ENDIF
      \STATE If $n=N_\Gamma+1$, return $y$; otherwise set $n\gets n+1$.
      \STATE $z\gets\Center(K)$; if it signals failure, return $y$.
      \STATE $(y,U,a,G,\eta,\xi)\gets
        \Step(y,U,z,\Gamma;A_\Gamma,\tau_\Gamma,\eta_\Gamma)$.
      \COMMENT{valid cut $a\le f$, cut-normal norm $G$, errors $\eta,\xi$}
      \STATE $m(u)\gets\max\{m(u),a(u)\}$ for $u\in\Q$.
    \ENDWHILE
    \STATE $\ell\gets\min_{u\in\Q}m(u)$.
  \ENDWHILE
  \STATE \textbf{return} $\widehat y\gets y$.
\end{algorithmic}
\end{breakablealgorithm}

The step $\Step(y,U,z,\Gamma;A,\tau,\eta)$ uses the
bracketed envelope line search of
\cref{lem:envelope-line-search}: at each trial center it calls the certified
proximal bundle solver of \cref{prop:bundle-solve}, then returns the feasible
upper witness and aggregate cut described in \cref{lem:feasible-coupling}.
The step retains the old witness whenever it retains the old upper
certificate. The displayed pseudocode uses exact model solves;
\cref{prop:numerical-implementation} specifies the certified numerical
replacements for its solves and emptiness tests. The query guard applies
inside every subroutine; an interrupted call retains the previous witness.
The selector uses the bounded-work implementation of
\citet{martinezRubioGuzman2026stable}, as specified in \cref{sec:implementation}.
The certified gap guarantee holds with probability at least $1-\delta$.
Thus the upper certificate $U$, phase scale $\Gamma$, cut-normal norm $G$,
descent inexactness $\eta$, and cut inexactness $\xi$ are the only quantities
passed from the analytic module to the movement argument.

\section{The abstract coupled phase}

Maintain a max-affine bundle model $m_t\le f$ on $\Q$, a lower certificate
$\ell_t=\min_\Q m_t$, a feasible upper witness $y_t$, and an upper certificate
$U_t\ge f(y_t)$.  A phase begins with scale
$\Gamma=(U_0-\ell_0)/2$.  Its body is
\begin{equation*}
  K_t=\set{u\in\Q:m_t(u)\le U_t-\Gamma}.
\end{equation*}
The bodies are nested because $m_t$ increases and $U_t$ decreases.

\newtarget{def:eta}{}
The coupled step allows two geometrically different errors.  The descent
error $\eta_t$ is vertical slack in the upper-bound decrease: it is the
amount lost from the ideal decrease $\mathcal D(G_t)$ for a specified
nonnegative descent profile $\mathcal D$.  The cut error $\xi_t$
is vertical slack in the separation at $z_t$: for cut normal $g_t$, it makes
the retained halfspace shallower by $\xi_t/\norm{g_t}_{s^\ast}$.  Thus
$\eta_t$ spends the phase's progress budget, whereas $\xi_t$ spends its cut
depth. An implementation obtains descent and cuts with some errors $\xi_t$ $\eta_t$, and we show that the rate of convergence is robust to it. 

\begin{definition}[Coupled step with qualified errors]
\label{def:coupled-step}
Fix a descent profile $\mathcal D:[0,\infty)\to[0,\infty)$.
At center $z_t\in K_t$, a coupled step returns an upper certificate
$U_{t+1}\le U_t$, a valid affine cut
$a_t(u)=b_t+\ip{g_t}{u}$, cut-normal norm
$G_t=\norm{g_t}_{s^\ast}$, descent inexactness $\eta_t\ge0$, and cut
inexactness $\xi_t\ge0$, such that
\begin{align}
  U_{t+1}
  &\le U_t-\mathcal D(G_t)+\eta_t,
  \label{eq:coupled-descent}\\
  a_t(z_t)-(U_{t+1}-\Gamma)
  &\ge\Gamma+\mathcal D(G_t)-\xi_t.
  \label{eq:coupled-cut}
\end{align}
The movement norm is $\ell_s$ and the cut normal is measured in
$\ell_{s^\ast}$.
\end{definition}

\begin{lemma}[Cut depth and inverse normal norm]
\label{lem:inverse-gradient-norm}\linktoproof{lem:inverse-gradient-norm}
Let $G_t$ be the cut-normal norm and $\xi_t$ the cut inexactness from
\cref{def:coupled-step}.  If $\xi_t\le\Gamma/2$ and the next body is nonempty,
then
\begin{equation*}
  \norm{z_{t+1}-z_t}_s\ge\frac{\Gamma}{2G_t}.
\end{equation*}
\end{lemma}
Note that if $G_t=0$, we can stop the algorithm, so zero never occurs in an inverse-norm sum.

\begin{theorem}[Movement-to-optimization reduction]
\label{thm:master-phase}\linktoproof{thm:master-phase}
Let $1<\kappa\le2$ and $\rho>0$. Suppose a selector for nested convex
sets in $R B_p^d$ satisfies
$\sum_{t=0}^{N-1}\norm{z_{t+1}-z_t}_s
=\widetilde O(RN^{1-\rho})$.
Consider the level updates of \cref{alg:holder-main}, using coupled steps
as in \cref{def:coupled-step} with profile $\mathcal D=\Dholder$,
with a budget of $T\ge1$ coupled steps in place of the accuracy and
first-order-call stopping rules. Stop earlier only if the certified gap vanishes.
Assume that, in each phase of scale $\Gamma$, the errors satisfy
$\sum_{t=0}^{N-1}\eta_t\le\Gamma/16$ and $0\le\xi_t\le\Gamma/2$
for every nonterminal prefix of length $N$.
From an initial certified gap $O(LR^\kappa)$, the retained upper witness
after at most $T$ coupled steps has optimization error at most
\begin{equation*}
  \widetilde O_\kappa\!\left(
    \frac{LR^\kappa}{T^{\kappa-1+\kappa\rho}}
  \right).
\end{equation*}
\end{theorem}

\section{Module I: a global minimizer lies in the auxiliary ball}

Assume in this section that $f:\R^d\to\R$ and that a global minimizer
$x^\ast$ satisfies $x^\ast\in\Q$.  Thus $\nabla f(x^\ast)=0$, and a value
at an infeasible point is still a valid upper bound on $f(x^\ast)$.

Given a current upper witness $y$ and movement center $z$, minimize $f$ on
the segment $[y,z]$:
\begin{equation*}
  x\in\argmin_{w\in[y,z]}f(w),
  \qquad g=\nabla f(x),
  \qquad G=\norm{g}_{q^\ast}.
\end{equation*}
Set
\begin{equation}\label{eq:ambient-step}
  y^+=x-(G/L)^{1/(\kappa-1)}v_q(g),
  \qquad
  U^+=f(x)-\Dholder(G),
\end{equation}
and use the tangent
\begin{equation*}
  a(u)=f(x)+\ip{g}{u-x}.
\end{equation*}

\begin{lemma}[Exact ambient coupling]
\label{lem:ambient-coupling}\linktoproof{lem:ambient-coupling}
Let $U$ be the current upper certificate, $G$ the gradient norm in
\cref{eq:ambient-step}, and $\Gamma$ the phase scale.  Then
\begin{equation*}
  f(y^+)\le U^+\le U-\Dholder(G),
  \qquad
  a(z)-(U^+-\Gamma)\ge\Gamma+\Dholder(G).
\end{equation*}
Thus the ambient module satisfies \cref{def:coupled-step}
with $\mathcal D=\Dholder$ and zero descent
and cut errors.
\end{lemma}

\begin{theorem}[Ambient rate with exact line search]
\label{thm:ambient-rate}\linktoproof{thm:ambient-rate}
Let $p<\min\{q,2\}$, let $f:\R^d\to\R$ be convex and differentiable and satisfy
\cref{eq:holder-smoothness}, and suppose a global minimizer obeys
$\norm{x^\ast}_p\le R$.  Use $R B_p^d$ only for the bundle lower model and
the movement selector, and grant the segment-minimization oracle used
in \cref{eq:ambient-step}.  Then after $T$ coupled outer steps the ambient
module returns a point $y_T\in\R^d$, not necessarily in $R B_p^d$, with
\begin{equation*}
  f(y_T)-f(x^\ast)
  \le
  \widetilde O_{\kappa,p,q}\!\left(
    \frac{LR^\kappa}{T^{\kappa(1+\rho_{p,q})-1}}
  \right).
\end{equation*}
\end{theorem}

The theorem explains the cleanest use of an externally supplied radius
bound.  The auxiliary ball restricts the comparator and makes movement
finite; it does not restrict the descent path or output.  This argument is
valid only when the optimization problem itself is unconstrained.  An
approximate one-dimensional search can replace the ideal oracle by assigning
its value and sign errors to the budgets in \cref{def:coupled-step}; the
feasible construction below gives the fully specified polynomial algorithm
used in the main theorem.

\section{Why constrained feasibility requires a new module}

For $\min_{\Q}f$, the point $y^+$ in \cref{eq:ambient-step} can be infeasible,
and the constrained minimizer can have nonzero gradient.  Projecting $y^+$
back to $\Q$ does not preserve either the H\"older decrease or the tangent-cut
alignment.  A direct $\kappa$-power proximal envelope suggests the same descent
profile, but we do not establish the required linear-time proximal-solve
bound for that construction.

We instead use the ``inexact gradient trick'' at the current phase accuracy, a classical technique from \citet{devolder2014first}.
\newtarget{def:mean-width}{}
For a tolerance $\omega>0$, define
\begin{equation}\label{eq:H-omega}
  H_\omega
  \defi
  \begin{cases}
    \displaystyle
    \left(\frac{2-\kappa}{2\kappa}\right)^{(2-\kappa)/\kappa}
    L^{2/\kappa}\omega^{-(2-\kappa)/\kappa},&1<\kappa<2,\\[3mm]
    L,&\kappa=2.
  \end{cases}
\end{equation}

\begin{lemma}[Inexact gradient trick]
\label{lem:quadraticization}\linktoproof{lem:quadraticization}
For every $u,v\in\Q$,
\begin{equation*}
  f(v)
  \le f(u)+\ip{\nabla f(u)}{v-u}
  +\frac{H_\omega}{2}\norm{v-u}_q^2+\omega.
\end{equation*}
\end{lemma}

For the feasible construction, use $s=\min\{q,2\}$. %
For a center $c\in\Q$ and coefficient $A>0$, define the quadratic envelope
\begin{equation*}
  \Menv(c)
  \defi
  \min_{u\in\Q}\set{f(u)+\frac A2\norm{u-c}_s^2}.
\end{equation*}
Let $J_s(w)=\nabla(\norm{w}_s^2/2)$ denote the duality map.

\newtarget{def:threshold}{}
\begin{definition}[Certified envelope evaluation]
\label{def:envelope-certificate}
At center $c$, a bundle call with tolerance $\tau$ returns a convex max-affine model
$\model\le f$ and points $x,y\in\Q$ such that
\begin{align}
  x&\in\argmin_{u\in\Q}
    \set{\model(u)+\frac A2\norm{u-c}_s^2},
  \label{eq:bundle-lower-point}\\
  f(y)+\frac A2\norm{y-c}_s^2
  &\le
  \model(x)+\frac A2\norm{x-c}_s^2+\tau.
  \label{eq:bundle-envelope-gap}
\end{align}
The aggregate residual and cut are
\begin{equation}\label{eq:aggregate-cut}
  g=A J_s(c-x),
  \qquad
  a(v)=\model(x)+\ip{g}{v-x}.
\end{equation}
\end{definition}

\begin{proposition}[Certified bundle solve]
\label{prop:bundle-solve}\linktoproof{prop:bundle-solve}
Let $1<s\le2$, let $\Q\subseteq R B_s^d$ be compact and convex, and let
$f$ be convex and differentiable on a neighborhood of $\Q$.
Fix $c\in\Q$, $A,\tau>0$, and $0<\omega\le\tau/2$. Suppose that
\begin{equation*}
  f(v)\le f(u)+\ip{\nabla f(u)}{v-u}
    +\frac{H_\omega}{2}\norm{v-u}_s^2+\omega
  \qquad(u,v\in\Q).
\end{equation*}
\newtarget{def:mu}{}
    Set $\mu=A(s-1)$. Repeating \cref{eq:bundle-lower-point} yields \cref{eq:bundle-envelope-gap} using
\begin{equation}\label{eq:bundle-complexity}
  O\!\left(
    \left(1+\frac{H_\omega}{\mu}\right)
    \log\left(2+\frac{H_\omega R^2}{\tau}\right)
  \right)
\end{equation}
first-order calls. %
\end{proposition}
Exact model solves are used above. We note that \cref{prop:numerical-implementation} gives the
same order of calls with certified approximate solves when $H_\omega/\mu=O_s(1)$.
The proof contracts the global model gap directly. At the coefficient
used below, $H_\omega/\mu=O_s(1)$, so this nonaccelerated inner solver
already has logarithmic oracle cost.

\section{The feasible envelope coupling}

Let $y^{\mathrm{old}}$ be the current feasible upper witness and $z$ the movement
center.  We evaluate the envelope along
\begin{equation*}
  c_\lambda=y^{\mathrm{old}}+\lambda(z-y^{\mathrm{old}}),
  \qquad 0\le\lambda\le1.
\end{equation*}
At each queried $c_\lambda$, compute the certified residual $g_\lambda$ and
the scalar
\begin{equation*}
  \psi_\lambda=\ip{g_\lambda}{z-y^{\mathrm{old}}}.
\end{equation*}

\begin{lemma}[Certified envelope line search]
\label{lem:envelope-line-search}\linktoproof{lem:envelope-line-search}
Using envelope evaluations with bundle tolerance $\tau$, sign bisection
returns a center $c\in[y^{\mathrm{old}},z]$ and certified residual $g$ such that
\begin{equation*}
  \Menv(c)\le\Menv(y^{\mathrm{old}})+\eta+\tau,
  \qquad
  \ip{g}{z-c}\ge0.
\end{equation*}
The number of envelope evaluations is
$O(\log(2+AR^2/\eta))$.
\end{lemma}

By \cref{lem:feasible-coupling}, the envelope-value bound preserves telescoping
descent up to budgeted errors, while $\ip{g}{z-c}\ge0$ keeps the cut deep at $z$.
Only selector centers enter the movement sum, not line-search trials.
The nested-body movement bound therefore still applies.

\begin{lemma}[Feasible descent and cut]
\label{lem:feasible-coupling}\linktoproof{lem:feasible-coupling}
Let $U$ be the old feasible upper certificate, let $\Gamma$ be the phase
scale, and let $(x,y,\model)$ be the bundle certificate at the line-search
center $c$.  Put $G=\norm{g}_{s^\ast}$.  The aggregate affine function in
\cref{eq:aggregate-cut} satisfies $a\le f$ on $\Q$, and the new feasible upper
certificate $U^+=\min\{U,f(y)\}$ obeys
\begin{align}
  U^+
  &\le U-\frac{G^2}{4A}+C_s(\eta+\tau),
  \label{eq:feasible-descent}\\
  a(z)-(U^+-\Gamma)
  &\ge\Gamma+\frac{G^2}{2A}-C_s(\eta+\tau).
  \label{eq:feasible-cut}
\end{align}
Its witness is $y^{\mathrm{old}}$ if $U\le f(y)$, and $y$ otherwise.
The cut is retained in both cases. Here $g$ is an aggregate residual,
not necessarily a subgradient of $f$. It is a $\tau$-subgradient of
$\Menv$ by \cref{eq:envelope-epsilon-subgradient}.
\end{lemma}

This is \cref{def:coupled-step} with $\mathcal D(G)=G^2/(4A)$ and
errors $C_s(\eta+\tau)$. Its numerical version in
\cref{prop:numerical-implementation} uses $\mathcal D(G)=G^2/(8A)$.
Before upper
progress terminates the phase, the descent inequality gives
$\sum G_t^2\lesssim A\Gamma$.  The cut inequality and movement give
$\Gamma\sum1/G_t=\widetilde O(RN^{1-\rho})$.  H\"older with exponents
$3$ and $3/2$ yields the quadratic phase bound
\begin{equation}\label{eq:quadratic-phase}
  N^{1+2\rho}
  =\widetilde O_s\!\left(\frac{AR^2}{\Gamma}\right).
\end{equation}

Choose planned phase length $N_\Gamma$ and parameters
\begin{equation}\label{eq:phase-parameters}
  \tau_\Gamma\asymp\eta_\Gamma\asymp\frac{\Gamma}{N_\Gamma},
  \qquad
  A_\Gamma\asymp_s
  L^{2/\kappa}
  \left(\frac{N_\Gamma}{\Gamma}\right)^{(2-\kappa)/\kappa}.
\end{equation}
Substitution in \cref{eq:quadratic-phase} gives
\begin{equation}\label{eq:fixed-point-result}
  N_\Gamma^{\kappa-1+\kappa\rho}
  =\widetilde O_{\kappa,s}\!\left(\frac{LR^\kappa}{\Gamma}\right).
\end{equation}
Thus the feasible quadratic module has exactly the H\"older exponent found in
the ambient warm-up.

\begin{theorem}[Feasible movement-coupled rate for $p<\min\{q,2\}$]
\label{thm:feasible-new-branch}\linktoproof{thm:feasible-new-branch}
Let $1\le p<\min\{q,2\}$ and $1<q\le\infty$, let
$\Q=R B_p^d$, and let $f$ be convex and differentiable on a neighborhood of
$\Q$, satisfying \cref{eq:holder-smoothness}.  Use $s=q$ when $q\le2$ and
$s=2$ when $q\ge2$.  Let $\rho_{p,q}$ be the movement exponent in
\cref{eq:rho-holder},
let $U_t$ denote the feasible upper certificates, and let
$\tau_\Gamma,\eta_\Gamma$ be the bundle and line-search tolerances in
\cref{eq:phase-parameters}. For a first-order budget $T\ge1$
and confidence parameter $\delta\in(0,1)$, run \cref{alg:holder-main}
with initialization from \cref{lem:initialization} and the target accuracy
and query budget from \cref{eq:accuracy-from-budget}, or use its
small-budget fallback. This returns
$\widehat x_T\in R B_p^d$, with $f^\star=\min_{u\in\Q}f(u)$, satisfying
\begin{equation*}
  f(\widehat x_T)-f^\star
  \le
  \widetilde O_{\kappa,p,q}\!\left(
    \frac{LR^\kappa}{T^{\kappa(1+\rho_{p,q})-1}}
  \right).
\end{equation*}
The bound holds with probability at least $1-\delta$,
using at most $T$ first-order calls, including initialization and all
inner bundle and line-search calls.
\end{theorem}

\section{Regime assembly and comparison with lower bounds}

The feasible theorem is assembled as follows.
\begin{enumerate}[leftmargin=2.2em]
  \item If $p<q\le2$, use $s=q$ in the envelope and the
  $p/q$ movement selector from the companion paper.

  \item If $p<2\le q$, norm comparison gives
  \begin{equation*}
    \norm{\nabla f(v)-\nabla f(u)}_2
    \le\norm{\nabla f(v)-\nabla f(u)}_{q^\ast}
    \le L\norm{v-u}_2^{\kappa-1}.
  \end{equation*}
  Use the Euclidean envelope and the $p/2$ movement selector.  Its exponent is
  $\rho_{p,2}=1/p$, which is the required $q\ge2$ exponent.

\end{enumerate}

For $p<\min\{q,2\}$, the powers reported by
\citet{guzman2015open}, based on the lower-bound framework of
\citet{guzmanNemirovski2015lower}, are
\begin{equation*}
  \kappa\left(\frac32+\frac1p-\frac1q\right)-1
  \quad(q<2),
  \qquad
  \kappa\left(1+\frac1p\right)-1
  \quad(q\ge2).
\end{equation*}
These agree with \cref{thm:holder-main}.

\section{Implementation and scope}
\label{sec:implementation}

The reduction assumes a selector that returns feasible centers with
movement $\widetilde O(RN^{1-\rho})$ on the generated nested bodies.
A bounded-work implementation on these bodies,
with the stated movement bound conditionally on the preceding history,
is supplied by \cref{prop:selector-implementation}, making use of the selector's implementation in
\citet[Theorem~9, Proposition~11, and Appendix~E]{martinezRubioGuzman2026stable}.

We work in the real-arithmetic model: arithmetic, comparisons,
logarithms, rational powers, powers with the fixed norm and regularity
exponents, and independent scalar standard Gaussian draws have unit
cost. Exact first-order calls are counted separately. For fixed
$\kappa,p,q$, the additional work is polynomial in
$d$, $1+LR^\kappa/\eps$, $\log(1/\delta)$, and
$\log(2+\norm{\nabla f(0)}_{s^\ast}/(LR^{\kappa-1}))$
in the new branch. The last dependence comes only from solving the
inner tangent models; the first-order bound is independent of affine
terms in $f$. The convex-solve and sampling bounds in
\cref{prop:selector-implementation}, together with the hard budgets in
\cref{alg:holder-main}, bound the work on every run.

\acks{
Crist\'obal Guzm\'an was partially funded by ANID FONDECYT 1251029 grant, and ANID Basal FB210017 National Center for Artificial Intelligence CENIA.
David Martínez-Rubio was funded by grant La Caixa Junior Leader Fellowship 2025. He thanks OpenAI for free access to their models. 
Mathieu Molina received funding from the European Research Council (ERC) under the European Union's Horizon Europe program (grant agreement No. 101170373), as a postdoctoral fellow at Tel Aviv University.

This work was elaborated in combination with ChatGPT/Codex. After providing the initial ideas (the contents of the warm-up \cref{sec:warmup} along with pointing to using the Moreau envelope analysis, which can absorb errors, plus treating Hölder smooth functions as inexact smooth functions), ChatGPT 5.6 (the newest model when we elaborated the results) was able to craft the details of the proof. 
This draft is not in the form we would have liked to share it and the authors will shortly improve its writing quality, organization of ideas and overall presentation. 
}

\clearpage
\appendix
\ifusecolttemplate\crefalias{section}{appendix}\fi

\section{Proof of the abstract phase theorem}

\newtarget{def:distance}{}
\begin{proof}\linkofproof{lem:inverse-gradient-norm}
The new body lies in the halfspace
\begin{equation*}
  H_t=\set{u:a_t(u)\le U_{t+1}-\Gamma}.
\end{equation*}
If $G_t=0$, the coupled cut is constant and its retained halfspace is empty.
Suppose $G_t>0$.  A displacement $w$ from $z_t$ into $H_t$ must satisfy
$\ip{g_t}{w}\le-[a_t(z_t)-(U_{t+1}-\Gamma)]$.  H\"older's inequality gives
$\norm{w}_s\ge[a_t(z_t)-(U_{t+1}-\Gamma)]/G_t$.  Equality is attained by a
suitably scaled negative unit $\ell_s$ direction on which $g_t$ attains its
dual norm.  Therefore the distance from the current center is
\begin{equation*}
  \dist_s(z_t,H_t)
  =\frac{a_t(z_t)-(U_{t+1}-\Gamma)}{G_t}
  \circled{1}[\ge]
  \frac{\Gamma-\xi_t}{G_t}
  \circled{2}[\ge]
  \frac\Gamma{2G_t}.
\end{equation*}
Here $\circled{1}$ uses the coupled cut
\cref{eq:coupled-cut} and $\mathcal D\ge0$, and $\circled{2}$ uses
$\xi_t\le\Gamma/2$.  Since $z_{t+1}$ belongs to the new body, its movement
is at least this distance.
\end{proof}

\begin{proof}\linkofproof{thm:master-phase}
Fix a phase of scale $\Gamma$ and a nonterminal prefix of
length $N$. We first bound $N$ and then sum the phase costs.
Recall that $U_t$ is the upper certificate and $G_t$ the cut-normal norm.
Here the descent profile is $\Dholder$. As
long as upper progress has not terminated the phase,
$U_0-U_N<\Gamma/4$.  Summing \cref{eq:coupled-descent} and using the descent
error budget gives
\begin{equation}\label{eq:master-power-sum}
  \sum_{t<N}G_t^{m_\kappa}
  \le C_\kappa L^{1/(\kappa-1)}\Gamma.
\end{equation}
Unpack the assumed soft movement bound as
$\sum_{t<N}\norm{z_{t+1}-z_t}_s
\le CR\Lambda(d,N)N^{1-\rho}$, where $\Lambda$ is nondecreasing and
polylogarithmic.  The inverse-normal-norm lemma and this movement bound give
\begin{equation}\label{eq:master-inverse-sum}
  \frac\Gamma2\sum_{t<N}\frac1{G_t}
  \le R\Lambda(d,N)N^{1-\rho}.
\end{equation}
Apply H\"older in the form
\begin{align*}
  N
  &=\sum_{t<N}
    (G_t^{m_\kappa})^{1/(m_\kappa+1)}
    (G_t^{-1})^{m_\kappa/(m_\kappa+1)}\\
  &\le
  \left(\sum_{t<N}G_t^{m_\kappa}\right)^{1/(m_\kappa+1)}
  \left(\sum_{t<N}\frac1{G_t}\right)^{m_\kappa/(m_\kappa+1)}.
\end{align*}
Substituting \cref{eq:master-power-sum,eq:master-inverse-sum}, raising to
$m_\kappa+1$, and then raising to $\kappa-1$ yields
\begin{equation*}
  N^{\kappa-1+\kappa\rho}
  \le C_\kappa
  \frac{LR^\kappa\Lambda(d,N)^\kappa}{\Gamma}.
\end{equation*}

It remains to verify gap contraction.  If upper progress occurs, the upper
certificate drops by $\Gamma/4=(U_0-\ell_0)/8$, while the lower certificate
does not decrease; the new gap is at most $7(U_0-\ell_0)/8$.  If the body
$\{m\le U-\Gamma\}$ is empty, then
$\ell=\min_\Q m\ge U-\Gamma$, so the new gap is at most
$\Gamma=(U_0-\ell_0)/2$.
The bounds on phase costs grow geometrically as the gap
shrinks and are therefore dominated by the final phase, including at most
one terminal step per phase. Starting from a certified gap $O(LR^\kappa)$
and inverting the resulting bound on the total number $T$ of coupled
steps gives optimization error
$\widetilde O_\kappa(LR^\kappa/T^{\kappa-1+\kappa\rho})$
for the retained upper witness.
\end{proof}

\section{The ambient module}

\begin{proof}\linkofproof{lem:ambient-coupling}
Exact minimization on $[y,z]$ gives
\begin{equation}\label{eq:ambient-line-signs}
  f(x)\le f(y),
  \qquad
  \ip{g}{z-x}\ge0.
\end{equation}
Let $a=(G/L)^{1/(\kappa-1)}$.  By
\cref{eq:holder-descent-lemma} and the norming-vector identity,
\begin{align*}
  f(x-av_q(g))
  &\le f(x)-aG+\frac L\kappa a^\kappa
  =f(x)-\frac{\kappa-1}{\kappa}
    L^{-1/(\kappa-1)}G^{\kappa/(\kappa-1)}
  =U^+.
\end{align*}
Thus $f(y^+)\le U^+$.  The first sign in
\cref{eq:ambient-line-signs} gives
$U^+\le U-\Dholder(G)$.  Because $x^\ast$ is a global minimizer,
$U^+\ge f(y^+)\ge f(x^\ast)$, so this is a valid upper certificate.

For the tangent cut,
\begin{equation*}
  a(z)-(U^+-\Gamma)
  =\ip{g}{z-x}+\Dholder(G)+\Gamma
  \circled{1}[\ge]\Gamma+\Dholder(G).
\end{equation*}
Here $\circled{1}$ uses the second sign in
\cref{eq:ambient-line-signs}.
\end{proof}

\begin{proof}\linkofproof{thm:ambient-rate}
Apply \cref{lem:ambient-coupling,thm:master-phase} with the imported movement
exponent $\rho_{p,q}$.  To initialize, query at $0$.  Since
$\nabla f(x^\ast)=0$, \cref{eq:holder-smoothness} gives
\begin{equation*}
  \norm{\nabla f(0)}_{q^\ast}
  \le L\norm{x^\ast}_q^{\kappa-1}
  \le LR^{\kappa-1}.
\end{equation*}
The tangent range over $R B_p^d$ is at most a constant times
$R\norm{\nabla f(0)}_{p^\ast}\le LR^\kappa$, where $p<q$.  Hence the
initial certified gap is $O(LR^\kappa)$.  Geometric phases give the claimed
rate.  Nothing in the proof requires the upper witness to lie in the
auxiliary ball.
\end{proof}

\section{Inexact gradient trick and the certified bundle solve}

\begin{proof}\linkofproof{lem:quadraticization}
It suffices to prove, for every $t\ge0$,
\begin{equation*}
  \frac L\kappa t^\kappa
  \le\frac{H_\omega}{2}t^2+\omega.
\end{equation*}
For $1<\kappa<2$, the maximum of the left side minus $Ht^2/2$ occurs at
$t=(L/H)^{1/(2-\kappa)}$ and equals
\begin{equation*}
  \frac{2-\kappa}{2\kappa}
  L\left(\frac LH\right)^{\kappa/(2-\kappa)}.
\end{equation*}
The definition of $H_\omega$ makes this value equal to $\omega$.  Combining
the scalar inequality with \cref{eq:holder-descent-lemma} proves the claim.
At $\kappa=2$, it is the usual descent lemma.
\end{proof}

\begin{proof}\linkofproof{prop:bundle-solve}
We bound the gap between the best evaluated proximal objective and the
global model minimum. Put $h(u)=A\norm{u-c}_s^2/2$ and $F=f+h$.
Query at $c$ and let $m_0$ be its affine tangent. For $j\ge0$, solve
\begin{equation*}
  x_j\in\argmin_{u\in\Q}\{m_j(u)+h(u)\},
  \qquad \ell_j=m_j(x_j)+h(x_j),
\end{equation*}
query at $x_j$, and retain $V_j=\min_{0\le i\le j}F(x_i)$ with a
corresponding witness. Stop when $E_j:=V_j-\ell_j\le\tau$; otherwise
add the tangent at $x_j$ to form $m_{j+1}$.

Strong convexity and $m_{j+1}\ge m_j$ give
\begin{equation*}
  \frac\mu2\norm{x_{j+1}-x_j}_s^2
  \le \ell_{j+1}-\ell_j=:b_j.
\end{equation*}
The new model contains the tangent at $x_j$, so
\begin{align*}
  E_{j+1}
  &\le F(x_{j+1})-\ell_{j+1}\\
  &\circled{1}[\le]
    \frac{H_\omega}{2}\norm{x_{j+1}-x_j}_s^2+\omega\\
  &\le \frac{H_\omega}{\mu}b_j+\omega.
\end{align*}
Here $\circled{1}$ uses the assumed quadratic upper model.
Since $V_{j+1}\le V_j$, we have $b_j\le E_j-E_{j+1}$. With
$a=H_\omega/\mu$, this proves
\begin{equation*}
  E_{j+1}\le\frac{a}{1+a}E_j+\frac{\omega}{1+a},
  \qquad
  E_j\le\omega+\left(\frac{a}{1+a}\right)^j E_0.
\end{equation*}
The initial tangent at $c$ gives
$E_0=f(x_0)-m_0(x_0)\le2H_\omega R^2+\omega$.
Thus $E_j\le\tau$ within \cref{eq:bundle-complexity}.
The retained model and witness satisfy
\cref{eq:bundle-lower-point,eq:bundle-envelope-gap}. All minima are over
the original $\Q$, so the stopping certificate is global.
\end{proof}

For $1<s\le2$, $\norm{\cdot}_s^2/2$ is
$\sigma_s$-strongly convex with $\sigma_s=s-1$.
Taking $\omega=\tau/8$ and $A\ge C_sH_\omega$ makes
$H_\omega/[A(s-1)]=O_s(1)$, as required for logarithmic oracle cost.

\section{Envelope line search and feasible coupling}

\begin{proof}\linkofproof{lem:envelope-line-search}
Fix one certificate $(x,y,\model)$ at a queried center $c$ and define its
lower envelope
\begin{equation*}
  \underline M(v)
  =\min_{u\in\Q}\set{\model(u)+\frac A2\norm{u-v}_s^2}.
\end{equation*}
Because $\model\le f$, $\underline M(v)\le\Menv(v)$ for all $v$.  At $c$,
the bundle gap gives
\begin{equation*}
  \Menv(c)
  \le f(y)+\frac A2\norm{y-c}_s^2
  \le\underline M(c)+\tau.
\end{equation*}
Danskin's theorem \citep{danskin1967maxmin} gives
$g=A J_s(c-x)\in\partial\underline M(c)$.  Combining the three facts yields
the certified $\tau$-subgradient inequality
\begin{equation}\label{eq:envelope-epsilon-subgradient}
  \Menv(v)
  \ge\Menv(c)+\ip{g}{v-c}-\tau,
  \qquad v\in\Q.
\end{equation}

Write $d=z-y^{\mathrm{old}}$, $c_\lambda=y^{\mathrm{old}}+\lambda d$, and
$\psi_\lambda=\ip{g_\lambda}{d}$.  If $\psi_0\ge0$, return $c_0$.  If
$\psi_0<0$, query $1$; when $\psi_1\le0$, apply
\cref{eq:envelope-epsilon-subgradient} at $c_1$ with $v=c_0$ to get
$\Menv(c_1)\le\Menv(c_0)+\tau$, and return $c_1$.

It remains to treat $\psi_0<0<\psi_1$.  Bisect while maintaining endpoints
$\lambda_-,\lambda_+$ with
$\psi_{\lambda_-}\le0\le\psi_{\lambda_+}$, and stop
when $\lambda_+-\lambda_-\le\eta/G_{\mathrm{ls}}$, where
$G_{\mathrm{ls}}=4AR^2$.  The envelope is $G_{\mathrm{ls}}$-Lipschitz on this
segment: comparison with the prox minimizer at the other
center gives
\begin{equation*}
  \abs{\Menv(c')-\Menv(c)}
  \le2AR\norm{c'-c}_s.
\end{equation*}
The $\tau$-subgradient inequality at the left endpoint, evaluated at
$c_0$, gives
$\Menv(c_{\lambda_-})\le\Menv(c_0)+\tau$.  Therefore
\begin{equation*}
  \Menv(c_{\lambda_+})
  \le\Menv(c_0)+\tau+\eta.
\end{equation*}
Finally,
$\ip{g_{\lambda_+}}{z-c_{\lambda_+}}
=(1-\lambda_+)\psi_{\lambda_+}\ge0$.  The initial bracket has length one, so
the number of evaluations is $O(\log(2+AR^2/\eta))$.
\end{proof}

\begin{proof}\linkofproof{lem:feasible-coupling}
Let
\begin{equation*}
  d_x=\norm{c-x}_s,
  \qquad d_y=\norm{c-y}_s,
  \qquad G=\norm{g}_{s^\ast}=Ad_x.
\end{equation*}
Optimality in \cref{eq:bundle-lower-point} supplies
$g_\model\in\partial\model(x)$ with
$g-g_\model\in\Ncone_\Q(x)$.  Hence, for every $v\in\Q$,
\begin{equation*}
  \model(v)
  \ge\model(x)+\ip{g_\model}{v-x}
  \ge\model(x)+\ip{g}{v-x}=a(v).
\end{equation*}
Thus $a\le\model\le f$.

The model prox objective is $A\sigma_s$-strongly convex.  Since
$\model(y)\le f(y)$, the bundle gap implies
\begin{equation*}
  \frac{A\sigma_s}{2}\norm{y-x}_s^2\le\tau.
\end{equation*}
The triangle inequality therefore gives
\begin{equation}\label{eq:dy-lower}
  d_y^2
  \ge\frac12d_x^2-\frac{2\tau}{A\sigma_s}.
\end{equation}

The line search gives
$\Menv(c)\le\Menv(y^{\mathrm{old}})+\eta+\tau$ and
$\ip{g}{z-c}\ge0$.  Since
$\Menv(y^{\mathrm{old}})\le f(y^{\mathrm{old}})\le U$, the upper certificate obeys
\begin{align*}
  f(y)
  &\le\Menv(c)+\tau-\frac A2d_y^2\\
  &\le U+\eta+2\tau-\frac A2d_y^2\\
  &\circled{1}[\le]
  U-\frac{G^2}{4A}+C_s(\eta+\tau).
\end{align*}
Here $\circled{1}$ uses \cref{eq:dy-lower}.  Taking the minimum with the old
upper value proves \cref{eq:feasible-descent}.

For the cut, the bundle gap also gives
\begin{equation*}
  \model(x)-f(y)
  \ge\frac A2(d_y^2-d_x^2)-\tau.
\end{equation*}
Consequently,
\begin{align*}
  a(z)-(f(y)-\Gamma)
  &\ge
  \ip{g}{z-x}+\frac A2(d_y^2-d_x^2)+\Gamma-\tau\\
  &=\ip{g}{z-c}+\frac A2(d_x^2+d_y^2)+\Gamma-\tau\\
  &\ge\Gamma+\frac{G^2}{2A}-C_s(\eta+\tau).
\end{align*}
Replacing $f(y)$ by the no-larger $U^+$ only deepens the cut.  This proves
\cref{eq:feasible-cut}.
\end{proof}

\section{Certified numerical solves and cuts}

\newtarget{def:numerical-tolerance}{}
\begin{proposition}[Certified numerical coupling]
\label{prop:numerical-implementation}\linktoproof{prop:numerical-implementation}
Let $1<s\le2$, $\Q=RB_p^d\subseteq RB_s^d$, and
$A\ge C_sH_{\tau/8}$. Fix $\eta,\tau>0$ and
$0<\nu\le\min\{\tau,AR^2\}$. Certified approximate model solves and
the sign search of \cref{lem:envelope-line-search} produce a feasible
witness $y^+$, certificate $U^+\ge f(y^+)$, and a valid affine cut
$\widehat a\le f$ such that
\begin{align*}
  U^+&\le U-\frac{G^2}{8A}+C_s(\eta+\tau+\nu),\\
  \widehat a(z)-(U^+-\Gamma)
    &\ge\Gamma+\frac{G^2}{4A}-C_s(\eta+\tau+\nu),
\end{align*}
where $G$ is the cut-normal norm in $\ell_{s^\ast}$. The number of
first-order calls is
$O_s(\log(2+AR^2/\tau)\log(2+AR^2/\eta))$.
For outer tests, compute a certified interval $[L_m,V_m]$ of width
$\Gamma/16$ for $\min_\Q m$, and terminate the phase if
$L_m\ge U-5\Gamma/4$. Keeping the largest certified lower bound
preserves a gap contraction of at most $7/8$ per phase.
\end{proposition}

\begin{proof}\linkofproof{prop:numerical-implementation}
\newtarget{def:objective-tolerance}{}
Put $h(u)=A\norm{u-c}_s^2/2$ and $\mu=A(s-1)$. For each tangent
model $m_j$, compute a feasible $\widetilde x_j$ and lower bound $l_j$
with $m_j(\widetilde x_j)+h(\widetilde x_j)\le l_j+\zeta$.
Write $\ell_j=\min_\Q(m_j+h)$, so $l_j\le\ell_j$, retain
$V_j=\min_{i\le j}(f+h)(\widetilde x_i)$, and stop when
$V_j-l_j\le\tau$. Otherwise add the tangent at $\widetilde x_j$.
For the exact model minimizer $x_j$, strong convexity gives
\begin{equation*}
  \norm{\widetilde x_j-x_j}_s^2\le2\zeta/\mu,
  \qquad
  \norm{\widetilde x_{j+1}-\widetilde x_j}_s^2
  \le4(\ell_{j+1}-\ell_j+2\zeta)/\mu.
\end{equation*}
The second bound compares both points to $x_j$. Set
$E_j=V_j-\ell_j$, $\omega=\tau/8$, and $a=2H_\omega/\mu$.
The inexact gradient trick at the last tangent and $V_{j+1}\le V_j$ give
\begin{equation*}
  E_{j+1}\le a(\ell_{j+1}-\ell_j)+(2a+1)\zeta+\omega
  \le a(E_j-E_{j+1})+(2a+1)\zeta+\omega.
\end{equation*}
For $\zeta\le\tau/[16(1+a)]$, this contracts to at most $\tau/4$.
Since $E_0\le2H_\omega R^2+\omega+\zeta$ and
$V_j-l_j\le E_j+\zeta$, it takes
$O_s(\log(2+AR^2/\tau))$ calls to certify
\cref{eq:bundle-envelope-gap} relative to the exact model minimum;
the point computed is $\widetilde x$, while the exact minimizer $x$
is used only in the analysis below.

At termination, let $x$ be the exact model minimizer and compute
$g=AJ_s(c-\widetilde x)$ in the stated real-arithmetic model.
The duality map satisfies
$\norm{J_s(w)-J_s(w')}_{s^\ast}
\le C_sR^{2-s}\norm{w-w'}_s^{s-1}$ on $2RB_s^d$;
this follows from its coordinate formula and the $(s-1)$-H\"older
continuity of $t\mapsto |t|^{s-2}t$.
For $g^\circ=AJ_s(c-x)$, choose
\begin{equation*}
  \zeta\le\min\left\{
    \frac{\tau}{16(1+a)},\ \nu,\quad
    \frac\mu2\left(\frac{\nu}{4C_sAR^{3-s}}\right)^{2/(s-1)}
  \right\},
  \qquad E_g=\frac{\nu}{4R}.
\end{equation*}
Then $\norm{g-g^\circ}_{s^\ast}\le E_g$. The exact cut
$a^\circ(v)=m(x)+\ip{g^\circ}{v-x}$ is valid, while the model gap
and convexity of $h$ give
$0\le m(\widetilde x)-m(x)-\ip{g^\circ}{\widetilde x-x}\le\zeta$.
Thus the computable cut
$\widehat a(v)=m(\widetilde x)+\ip{g}{v-\widetilde x}-\zeta-2RE_g$
satisfies
\begin{equation}\label{eq:numerical-cut-comparison}
  a^\circ(v)-\zeta-4RE_g\le\widehat a(v)\le a^\circ(v)\le f(v)
  \qquad(v\in\Q).
\end{equation}

By \cref{eq:envelope-epsilon-subgradient}, $g$ is a
$(\tau+2RE_g)$-subgradient of $\Menv$ on $\Q$. The proof of
\cref{lem:envelope-line-search} therefore applies with that error,
giving $\Menv(c)\le\Menv(y^{\mathrm{old}})+\eta+\tau+2RE_g$ and
$\ip{g}{z-c}\ge0$ in $O(\log(2+AR^2/\eta))$ trials.
Consequently $\ip{g^\circ}{z-c}\ge-2RE_g$. Apply the proof of
\cref{lem:feasible-coupling} to $g^\circ$, then use
\cref{eq:numerical-cut-comparison} and
$\norm{g^\circ}_{s^\ast}^2\ge G^2/2-E_g^2$.
Since $E_g^2/A\le\nu/16$, this gives the stated descent and cut
bounds. Retain the witness corresponding to $U^+=\min\{U,f(y)\}$.

For the outer test, let $V_m=m(v_m)$ with $v_m\in\Q$.
If $L_m\ge U-5\Gamma/4$, the new gap is at most $5\Gamma/4$,
which is $5/8$ of the starting gap. Otherwise
$m(v_m)<U-19\Gamma/16$, leaving slack $3\Gamma/16$ in every
model cut of $K=\{m\le U-\Gamma\}$. Upper progress still contracts
the gap by $7/8$.
To apply the companion's implementation, normalize $R=L=1$ and
let $B$ be the largest norm of the stored cut normals. Here
$B\le2\max A$. With $\theta=\min\{1/2,\Gamma/[32(1+B)]\}$,
\begin{equation*}
  (1-\theta)v_m+\frac{\theta}{2\sqrt d}B_2^d\subseteq K.
\end{equation*}
This follows from the cut slack and $\norm{v_m}_p\le1$.
The phase parameters give inverse-polynomial radii and objective
tolerances for fixed $\kappa,s$.
The transfer to the companion's sample problems is
proved in \cref{prop:selector-implementation}.
Inner tangent slopes are bounded by
$\norm{\nabla f(0)}_{s^\ast}+1$ in normalized units, giving the
logarithmic dependence stated in \cref{sec:implementation}.
\end{proof}

\begin{proposition}[Feasible selector implementation]
\phantomsection\label{prop:selector-implementation}
Let $1\le p<\min\{q,2\}$ and $s=\min\{q,2\}$.
The numerical version of \cref{alg:holder-main}, initialized by
\cref{lem:initialization}, admits feasible randomized selectors with
$\ell_s$-movement $\widetilde O_{p,s}(RN_\Gamma^{1-\rho_{p,q}})$ in phase $k\ge0$,
conditionally on its initial history, with probability at least
$1-\delta_k$. Their implementation uses no additional first-order calls.
For $0<\eps\le LR^\kappa$, allocating
$\delta_k=\delta/[2(k+1)^2]$ gives simultaneous success with probability
at least $1-\delta$ and the additional-work bound in
\cref{sec:implementation} on every run. On a sampling cutoff, return
the retained feasible upper witness.
\end{proposition}

\begin{proof}
Normalize $R=L=1$ by replacing $f$ with $f(R\,\cdot)/(LR^\kappa)$.
Fix a phase and write $N=N_\Gamma$, $\rho=\rho_{p,q}$.
Apply \citet[Proposition~11]{martinezRubioGuzman2026stable} with their
$(q,T)=(s,N)$, sampling and solver errors
$E=E_{\mathrm{sol}}=N^{-\rho}/4$, and failure probability $\delta_k/2$.
Use exactly their sample count (17), objective tolerance (32), and
feasible empirical average. These give exact feasibility and, by their
path-error estimate (33) and Theorem~9, the stated movement bound.
Their conditional concentration argument already covers adaptive bodies
and fresh samples; shorter phases can be padded by their last body.
Our coupling is evaluated at the computed feasible center itself, so
no additional cut error is needed.

We check only the geometric input to the companion's Appendix~E.
The preceding proof supplies an interior ball of radius
$\theta/(2\sqrt d)$ in each queried $K_t$, where
$\theta=\min\{1/2,\Gamma/[32(1+B)]\}$ and the stored cut normals
satisfy $B\le2\max A$. Shrinking the head-tail decomposition of $v_m$
by $1-\theta$, exactly as there, gives an interior ball in its lift
of radius $\theta\min\{1,R_1,R_s\}/(4\sqrt d)$, with
$R_1=N^{1-1/p}$ and $R_s=N^{1/s-1/p}$; an outer radius is at most
$4N$. Norm constraints and stored cuts supply the same separation
oracles without queries to $f$. Hence the certified feasible solves
of Appendix~E apply to our aggregate cuts. The fixed powers are
permitted directly in our model, without rational-exponent approximation.

Also use the Gaussian cutoff in Appendix~E with planned draw budget
$(N+1)M$, where $M$ is the sample count above, and phase failure
parameter $\delta_k$. It costs probability at most $\delta_k/2$ and
bounds the work of every sample solve. Since $\Gamma>\eps/2$ in active
phases, our planned horizons and $A_\Gamma$ make all inverse radii,
inverse tolerances, and sample counts polynomial in the parameters of
\cref{sec:implementation}. The query guard bounds the number of stored
cuts and center calls by $J$, and phase contraction bounds the number
of phases. Model and tangent solves use the same certified procedure
on $\Q$, with the logarithmic affine-term dependence established above.
Thus work is bounded on every run, and
$\sum_{k\ge0}\delta_k\le\delta$ gives simultaneous success.
Rescaling restores $R,L$.
\end{proof}

\section{Phase tuning, initialization, and final assembly}

\begin{lemma}[Quadratic phase and fixed point]
\label{lem:quadratic-fixed-point}
Consider $N$ nonterminal transitions in one phase with fixed
$A,\Gamma>0$, nonempty bodies $K_0,\ldots,K_N\subseteq RB_p^d$, and
$U_0-U_N<\Gamma/4$. Suppose the selector satisfies
$\sum_{t=0}^{N-1}\norm{z_{t+1}-z_t}_s
=\widetilde O(RN^{1-\rho})$.
Let $G_t$ be the aggregate residual norms. Assume the cumulative
contribution of the line-search, bundle, and numerical errors is a
sufficiently small multiple of $\Gamma$, and each cut error is at most
$\Gamma/2$. Then
\begin{equation*}
  N^{1+2\rho}
  =\widetilde O_s\!\left(\frac{AR^2}{\Gamma}\right).
\end{equation*}
At the planned horizon $N=N_\Gamma$, with the parameters in \cref{eq:phase-parameters}, this is equivalent, up to
constants and logs, to \cref{eq:fixed-point-result}.
\end{lemma}

\begin{proof}
Before upper progress, sum \cref{eq:feasible-descent} to obtain
$\sum_{t<N}G_t^2\le CA\Gamma$. The weaker constants in \cref{prop:numerical-implementation} give the same estimate.  The cut depth from
\cref{eq:feasible-cut}, after absorbing its error, and the unpacked movement
bound
$\sum_{t<N}\norm{z_{t+1}-z_t}_s
\le CR\Lambda(d,N)N^{1-\rho}$ give
\begin{equation*}
  \frac\Gamma2\sum_{t<N}\frac1{G_t}
  \le R\Lambda(d,N)N^{1-\rho}.
\end{equation*}
H\"older with exponents $3$ and $3/2$ gives
\begin{equation*}
  N
  \le
  \left(\sum_{t<N}G_t^2\right)^{1/3}
  \left(\sum_{t<N}\frac1{G_t}\right)^{2/3}.
\end{equation*}
Substitution proves the quadratic phase inequality.

To test the planned horizon, set $N=N_\Gamma$ and use
$A\asymp L^{2/\kappa}(N/\Gamma)^{(2-\kappa)/\kappa}$.  The phase inequality
becomes
\begin{equation*}
  N^{1+2\rho-(2-\kappa)/\kappa}
  \lesssim
  L^{2/\kappa}R^2\Gamma^{-2/\kappa}.
\end{equation*}
Raising to $\kappa/2$ yields
\begin{equation*}
  N^{\kappa-1+\kappa\rho}
  \lesssim\frac{LR^\kappa}{\Gamma},
\end{equation*}
with the movement logarithms restored as in
\cref{eq:fixed-point-result}.
\end{proof}

To specify the horizon in \cref{alg:holder-main}, fix
$\Lambda(d,N)=1+\log(2d+2)\log(2N)$.
The selector estimates in \citet{martinezRubioGuzman2026stable}, including their
feasible approximations on the success event, give movement at most
$C_{p,s}R\Lambda(d,N)N^{1-\rho}$ over $N$ transitions.
Fix the constants in \cref{eq:phase-parameters} as in the proof below,
and a constant $C_{\kappa,p,s}$ large enough for the preceding phase bound.
The computable horizon rule is
\begin{equation}\label{eq:planned-phase-horizon}
  N_\Gamma=\min\left\{2^j:j\in\mathbb Z_{\ge0},\quad
    (2^j)^{\kappa-1+\kappa\rho}
    >C_{\kappa,p,s}\frac{LR^\kappa}{\Gamma}
      \Lambda(d,2^j)^\kappa\right\}.
\end{equation}
Thus $N_\Gamma$ counts transitions, and a phase uses at most
$N_\Gamma+1$ selector points.

\begin{lemma}[Affine-invariant initialization]
\label{lem:initialization}
One certified envelope call at $c=0$ with
\begin{equation*}
  A_0\asymp LR^{\kappa-2},
  \qquad
  \tau_0\asymp LR^\kappa
\end{equation*}
produces a feasible upper witness and a valid affine lower cut whose certified
gap over $\Q$ is $O_{\kappa,s}(LR^\kappa)$, independently of any affine term
in $f$.
\end{lemma}

\begin{proof}
Let $(x_0,y_0,\model_0)$ be the bundle certificate and
$g_0=A_0J_s(-x_0)$.  The bundle gap and $\norm{x_0}_s\le R$ give
\begin{equation*}
  f(y_0)-\model_0(x_0)
  \le\frac{A_0R^2}{2}+\tau_0.
\end{equation*}
Moreover $\norm{g_0}_{s^\ast}\le A_0R$.  The aggregate cut
$a_0(v)=\model_0(x_0)+\ip{g_0}{v-x_0}$ therefore satisfies
\begin{equation*}
  f(y_0)-\min_{v\in\Q}a_0(v)
  \le\frac52A_0R^2+\tau_0
  =O(LR^\kappa).
\end{equation*}
The bound on the initial tangent gap in \cref{prop:bundle-solve}
is independent of an affine term in $f$, so the first-order call bound
is affine-invariant as well. With numerical solves, the buffered cut and a lower-bound
tolerance $O(LR^\kappa)$ add only $O(LR^\kappa)$ to this initial gap.
\end{proof}

\begin{proof}\linkofproof{thm:feasible-new-branch}
For randomized selectors, assign phase $k\ge0$ failure probability
$\delta_k=\delta/[2(k+1)^2]$, using
\cref{prop:selector-implementation}, including its Gaussian cutoff.
Condition on simultaneous selector success. Fresh samples make each
phase guarantee valid conditionally on the past, so this event has
probability at least $1-\sum_{k\ge0}\delta_k\ge1-\delta$.
Initialize with \cref{lem:initialization}.  At a phase of scale $\Gamma$, set
$\tau_\Gamma,\eta_\Gamma$ to a sufficiently small constant times
$\Gamma/N_\Gamma$ and choose $A_\Gamma$ as in
\cref{eq:phase-parameters}.  Then the cumulative errors in
\cref{eq:feasible-descent,eq:feasible-cut} are small enough for
\cref{lem:quadratic-fixed-point}.  Choose $N_\Gamma$ by
\cref{eq:planned-phase-horizon}. If a phase contained $N_\Gamma$ nonterminal
transitions, \cref{lem:quadratic-fixed-point} would give the opposite
inequality, a contradiction.  Thus the phase terminates earlier; one final
empty-body step changes the count by at most one.

Take $\omega=\tau_\Gamma/8$ and a sufficiently large
$A_\Gamma\asymp_{\kappa,s}H_\omega$. By
\cref{prop:bundle-solve,prop:numerical-implementation}, every envelope
call is logarithmic in $2+AR^2/\tau_\Gamma$.
For numerical solves, take
$\nu_\Gamma=\tfrac12\min\{\tau_\Gamma,A_\Gamma R^2\}$ and use the
early lower-bound test in \cref{prop:numerical-implementation};
its contraction factor $5/8$ is also at most $7/8$.  The line search is logarithmic as well. Put
$\alpha=\kappa-1+\kappa\rho>0$. To reach error $\eps$, every active
phase has $\Gamma>\eps/2$. Reading the gaps backwards from the final
active phase bounds the oracle costs by a geometric series with ratio
$(7/8)^{1/\alpha}$, up to logarithms. Including initialization, choose
the deterministic first-order budget
\begin{equation}\label{eq:oracle-budget}
  \begin{aligned}
  J(\eps,\delta)&=\left\lceil C_{\kappa,p,q}
    \left(1+(LR^\kappa/\eps)^{1/\alpha}\right)
    \left[1+\log(2d)+\log(2+LR^\kappa/\eps)
      +\log(2/\delta)\right]^8\right\rceil,\\
  \alpha&=\kappa-1+\kappa\rho.
  \end{aligned}
\end{equation}
Here the fixed constant covers all the preceding bounds. The logarithmic
power $8$ suffices: $\rho>1/2$ gives $\kappa/\alpha\le2$, so the horizon
rule costs at most four logarithmic powers, the inner solves and line
search add two, and geometric phase summation preserves this order.
In particular, $J(\eps,\delta)=\widetilde O_{\kappa,p,q}
(1+(LR^\kappa/\eps)^{1/\alpha})$.
On simultaneous success, the certified gap reaches $\eps$
within this budget, and each phase uses at most $N_\Gamma+1$ selector
calls. On every run, the guards in \cref{alg:holder-main} enforce these
budgets and retain a feasible output.
Only within-phase edges are charged to movement; restarting the
selector introduces no charged edge between phases.
For a prescribed first-order budget $T\ge1$, set
\begin{equation}\label{eq:accuracy-from-budget}
  \eps_T=\min\left\{LR^\kappa 2^{-j}:j\in\mathbb Z_{\ge0},\quad
       J(LR^\kappa 2^{-j},\delta)\le T\right\},
  \qquad J=J(\eps_T,\delta).
\end{equation}
If this set is nonempty, run \cref{alg:holder-main} with accuracy $\eps_T$.
Minimality gives $J(\eps_T/2,\delta)>T$, while
$J(\eps_T,\delta)\le T$ bounds $\log(2+LR^\kappa/\eps_T)$ by
$O_{\kappa,p,q}(\log(2+T))$. Hence
$\eps_T=\widetilde O_{\kappa,p,q}(LR^\kappa/T^\alpha)$.
If the set is empty, use one query at $0$ and minimize its tangent over
$\Q$. This gives error at most $LR^\kappa/\kappa$; since then
$T<J(LR^\kappa,\delta)$, the same rate holds after increasing its
logarithmic factor.
Every retained upper witness is feasible and matches
its upper certificate.
\end{proof}

\begin{proof}\linkofproof{thm:holder-main}
If $p<q\le2$, use \cref{thm:feasible-new-branch} with $s=q$.  If
$p<2\le q$, use the same theorem with $s=2$ and the $p/2$ movement selector.
These two cases cover $p<\min\{q,2\}$ and give
\cref{eq:holder-main-rate}.
For $p=q\le2$, the classical universal accelerated method
\citep{nesterov2015universal} in the $p$-geometry gives error
$\widetilde O_{\kappa,p}(LR^\kappa/T^{3\kappa/2-1})$, including the
logarithmic loss at $p=q=1$. For this endpoint, take $r=1+1/\log(2d+2)$ and the prox function
$u\mapsto e^2\norm{u}_r^2/[2(r-1)]$, which is $1$-strongly convex
in $\ell_1$ and has range $O(R^2\log(2d))$ on $\Q$. Since
$\kappa(1+1/p-(1/p-1/2)_+)-1=3\kappa/2-1$ for $p\le2$,
this is also \cref{eq:holder-main-rate}.
\end{proof}

\clearpage
\begingroup
\sloppy
\printbibliography
\endgroup

\end{document}